%% file: IV-EBPRev.tex
\documentclass[bjps,preprint]{imsart}

\usepackage{graphicx, color, amsmath, amsfonts, amssymb, amsthm, amscd}
\usepackage{epsf}
\usepackage{psfrag}

\usepackage{appendix}

\usepackage[utf8]{inputenc}

\startlocaldefs

\newtheorem{theorem}{Theorem}[section]

\newtheorem{lemma}{Lemma}[section]
\newtheorem{proposition}{Proposition}[section]

\newtheorem{remark}{Remark}[section]

\newcommand{\sumtwo}[2]{\sum_{\substack{#1 \\ #2}}} 

\newcommand{\calA}{\mathcal{A}}

\newcommand{\calC}{\mathcal{C}}
\newcommand{\calD}{\mathcal{D}}

\newcommand{\calG}{\mathcal{G}}
\newcommand{\calH}{\mathcal{H}}

\newcommand{\calL}{\mathcal{L}}

\newcommand{\calP}{\mathcal{P}}

\newcommand{\calU}{\mathcal{U}}

\newcommand{\frg}{\mathfrak{g}}

\newcommand{\frq}{\mathfrak{q}}

\newcommand{\frt}{\mathfrak{t}}

\newcommand{\frG}{\mathfrak{G}}

\newcommand{\bbC}{\mathbb{C}}

\newcommand{\bbE}{\mathbb{E}}

\newcommand{\bbL}{\mathbb{L}}

\newcommand{\bbN}{\mathbb{N}}

\newcommand{\bbP}{\mathbb{P}}

\newcommand{\bbR}{\mathbb{R}}

\newcommand{\bbZ}{\mathbb{Z}}

\newcommand{\sfe}{{\sf e}}

\newcommand{\bfE}{{\ensuremath{\mathbf E}} }

\newcommand{\bfP}{{\ensuremath{\mathbf P}} }

\newcommand{\bfZ}{{\ensuremath{\mathbf Z}} }

\newcommand{\bfq}{{\ensuremath{\mathbf q}} }

\newcommand{\bft}{{\ensuremath{\mathbf t}} }

\newcommand{\Zd}{\bbZ^d}
\newcommand{\Rd}{\bbR^d}

\newcommand{\Prob}{\mathbb{P}}

\newcommand{\inprod}[2]{\langle#1,#2\rangle_d}
\newcommand{\setof}[2]{\left\{#1 \,:\, #2 \right\}}

\newcommand{\defby}{\stackrel{\text{\tiny{\rm def}}}{=}}

\newcommand{\IF}[1]{\mathbf{1}_{\{#1\}}}

\newcommand{\cone}{Y}

\newcommand{\define}[1]{\textsf{#1}}

\newcommand{\lb}{\left(}
\newcommand{\rb}{\right)}
\newcommand{\lbr}{\left\{}
\newcommand{\rbr}{\right\}}

\def\1{\ifmmode {1\hskip -3pt \rm{I}}
\else {\hbox {$1\hskip -3pt \rm{I}$}}\fi} 

\newcommand{\df}{\stackrel{\Delta}{=}}

\newcommand{\ie}{\textit{i.e.}}
\newcommand{\eg}{\textit{e.g.}}

\newcommand{\Znf}[1]{Z_n^{#1}}

\newcommand{\Pnf}[1]{\bbP_n^{#1}}

\newcommand{\Znx}{Z_{n,x}}

\newcommand{\xl}{\xi_\lambda}
\newcommand{\xln}{\xi_{\lambda_0}}

\newcommand{\Kl}{{\mathbf K}_\lambda}
\newcommand{\Km}{{\mathbf K}_\mu }
\newcommand{\Kln}{{\mathbf K}_{\lambda_0}} 
\newcommand{\intKln}{\mathring{\mathbf K}_{\lambda_0}}

\newcommand{\be}{\begin{equation}}
\newcommand{\ee}{\end{equation}}

\newcommand{\wq}{\mathsf{w}_{\lambda ,\beta}^\theta}
\newcommand{\wa}{\mathbf{w}_{\lambda ,\beta}}
\newcommand{\Zxf}[1]{Z_x^{#1}}
\newcommand{\ESRW}{\bbE_{\scriptscriptstyle\rm SRW}}
\newcommand{\PSRW}{\bbP_{\scriptscriptstyle\rm SRW}}
\endlocaldefs

\begin{document}

\begin{frontmatter}

\title{The Statistical Mechanics of Stretched Polymers}
\runtitle{The Statistical Mechanics of Stretched Polymers}

\begin{aug}

\author{\fnms{Dmitry}  \snm{Ioffe}\thanksref{a-di}\ead[label=e-di]{ieioffe@ie.technion.ac.il}}
  \and
  \author{\fnms{Yvan} \snm{Velenik}\thanksref{a-yv,t-yv}\ead[label=e-yv]{Yvan.Velenik@unige.ch}}

  \thankstext{t-yv}{Partially supported through Swiss National Science Foundation grant \#200020-121675.}

  \runauthor{D. Ioffe and Y. Velenik}

  \affiliation[a-di]{Technion}
  \affiliation[a-yv]{University of Geneva}

  \address{Faculty of Industrial Engineering and Management\\Technion\\Haifa 3200\\Israel.\\
  \printead{e-di}}

  \address{Department of Mathematics\\University of Geneva\\2-4, rue du Lièvre\\1211 Genève 4\\Switzerland.\\
  \printead{e-yv}}

\end{aug}

\begin{abstract}
We describe some recent results concerning the statistical properties of a self-interacting polymer stretched by an external force. We concentrate mainly on the cases of purely attractive or purely repulsive self-interactions, but our results are stable under suitable small perturbations of these pure cases. We provide in particular a precise description of the stretched phase (local limit theorems for the end-point and local observables, invariance principle, microscopic structure). Our results also characterize precisely the (non-trivial, direction-dependent) critical force needed to trigger the collapsed/stretched phase transition in the attractive case. We also describe some recent progress: first, the determination of the order of the phase transition in the attractive case; second, a proof that a semi-directed polymer in quenched random environment is diffusive in dimensions $4$ and higher when the temperature is high enough. In addition, we correct an incomplete argument from~\cite{IV08}.
\end{abstract}



\tableofcontents

\end{frontmatter}

\section{Introduction}
The statistical mechanics of stretched polymers has already a long history~(see, \eg, \cite{deGennes,RubinsteinColby}) and has known a strong revival in recent years, thanks in particular to remarkable experimental progress (in particular, the development of micromanipulation techniques, such as optical tweezers and atomic force microscopy, that allow a direct experimental realization of such a situation, making it possible, for example, to measure precisely the force/extension relation for given polymer chains).

\bigskip
In the present review, the polymers are always assumed to be long, flexible chains, which means that we are studying the polymer chains at a scale large compared to their persistence length. It should be noted, however, that our techniques should be able to handle a non-trivial persistence length, or even the case of a semiflexible polymer, the persistence length of which is comparable to its length.

\begin{figure}
\begin{center}
\includegraphics[width=6cm]{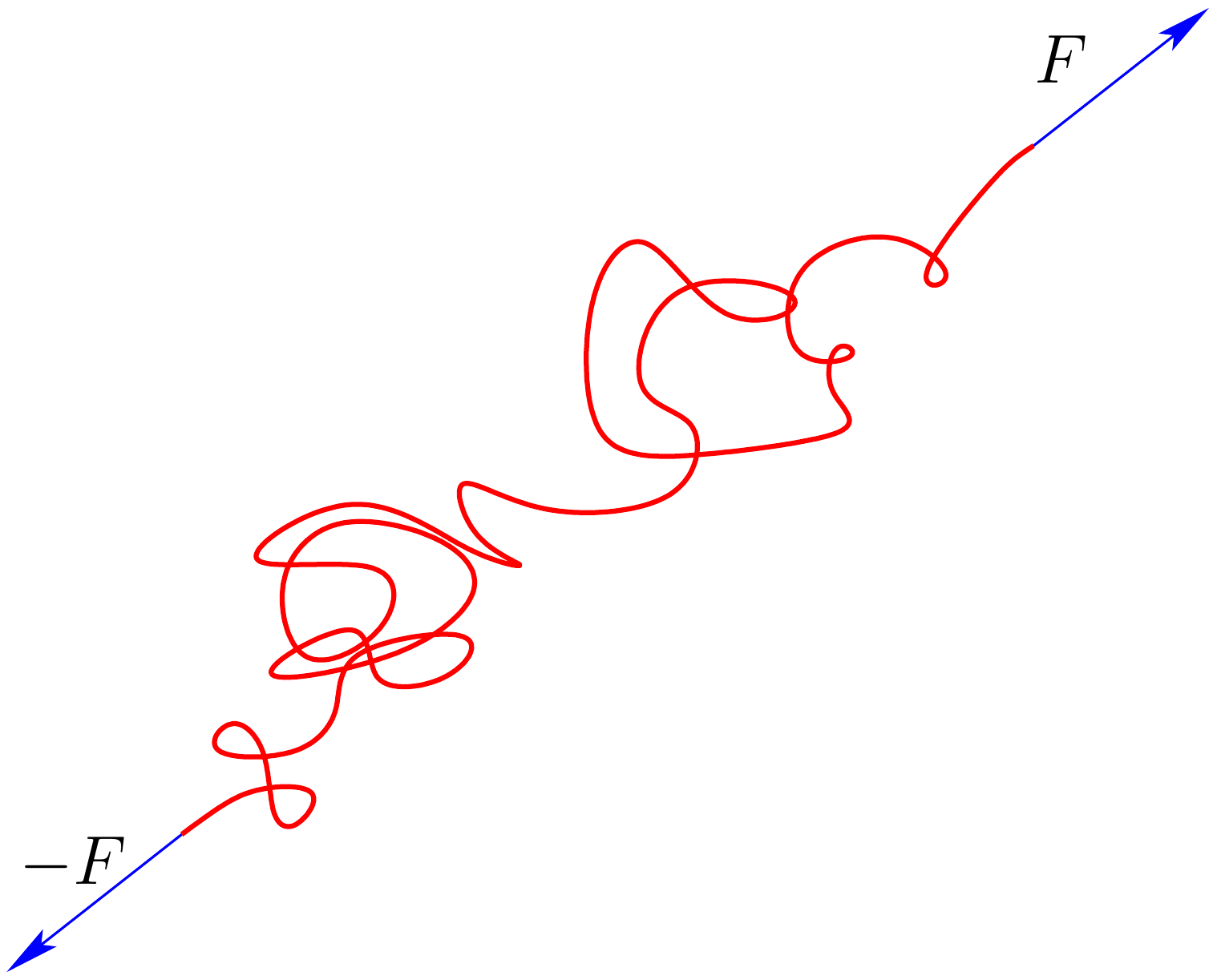}
\includegraphics[width=6cm]{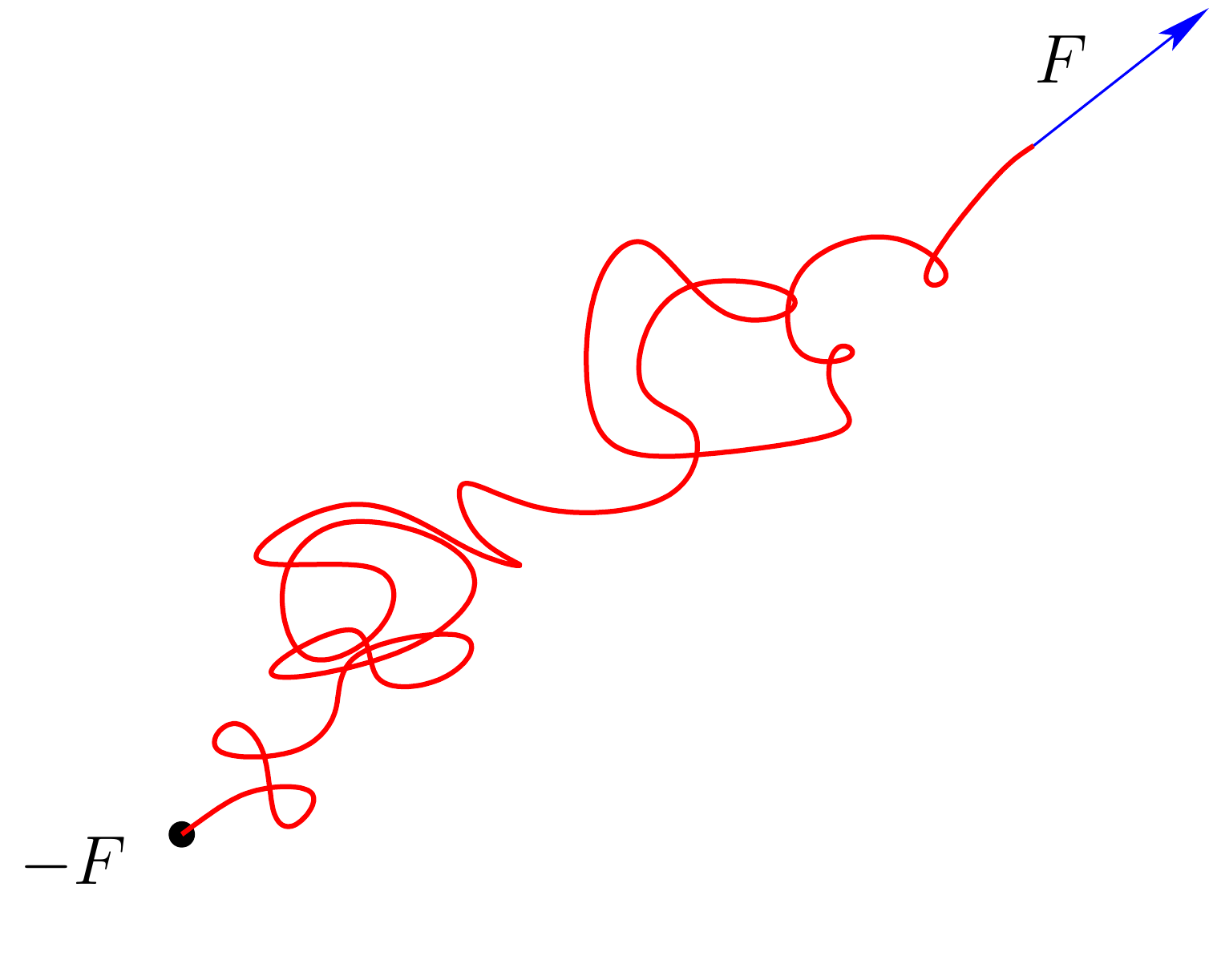}
\end{center}
\caption{Left: A polymer chain pulled at both end-points by opposite force $F$ and $-F$. Right: The equivalent setting considered in the present review, in which one of the end-point is pinned at the origin, the other one being pulled by a force $F$.}
\label{fig-setting}
\end{figure}
The physical situation we want to analyze is depicted in Figure~\ref{fig-setting}. Namely, we consider a long polymer chain, pulled by forces $F$ and $-F$ at its extremities. However, this setting is slightly inconvenient, as the spatial location of the polymer is not fixed. To lift the ambiguity, we fix one of the end-points at the origin, and apply a force $F$ to the other end-point. Obviously, the two descriptions are physically completely equivalent.

There are two contributions to the polymer energy: the first one is the internal energy due to the self-interaction, which will be denoted by $\Phi$; the second is the contribution due to the work done by the force, $-\inprod{F}{D}$, where $D$ is the total extension of the polymer (\ie, the position of its free end-point).

All this review, except Sections~\ref{sec_OrderTransition} and~\ref{sec_QuenchedDisorder} is based on the work~\cite{IV08}, in which proofs of the various statements made below can be found.

In Appendix~\ref{sec-newProof}, we provide an alternative to an incomplete argument we gave in~\cite{IV08}.

\paragraph{Acknowledgments.}
Y.V. is happy to thank Sacha Friedli and Bernardo de Lima for their wonderful organization of the XIIth Brazilan School of Probability, and the invitation to contribute this text.
This work was partially supported by Swiss NSF grant \#200020-121675.

\section{The model}
\subsection{Polymer configurations and their distribution}
A configuration of the polymer is given by a length $n$ nearest-neighbor path on $\Zd$, $\gamma=(\gamma(0),\ldots,\gamma(n))$, with $\gamma(0)=0$. Notice that there would be no problems in considering more general geometries (\eg, finite-range jumps).

To each configuration $\gamma$ of the polymer, we associate an internal energy $\Phi(\gamma)$ given by
$$
\Phi(\gamma) \defby \sum_{x\in\Zd} \phi(\ell_x(\gamma)),
$$
where $\ell_x(\gamma)$ is the local time of $\gamma$ at $x$, $\ell_x(\gamma) \defby \sum_{i=0}^{n} \IF{\gamma(i)=x}$. The function $\phi:\bbN\to\overline\bbR$ is a non-decreasing function satisfying $\phi(0)=0$.
\begin{remark}
Here, and in the following, we are considering local times at vertices. We could just as well have considered more general local potentials, \eg, local times through (possibly directed) edges.
\end{remark}
We shall consider two classes of interaction potentials $\phi$: attractive and repulsive. They are characterized as follows.
\begin{itemize}
\item \textbf{Repulsive self-interactions.} These are characterized by
$$
\phi(\ell_1+\ell_2) \geq \phi(\ell_1)+\phi(\ell_2).
$$
The terminology can be easily understood: The energetic cost of placing $\ell_1+\ell_2$ monomers at a given vertex $x$ is higher than that of placing $\ell_1$ monomers at a vertex $x$ and $\ell_2$ monomers at a different vertex $y$. Since lower energies are favored, this will induce a self-repulsion of the chain.
\item \textbf{Attractive self-interactions.} These are similarly characterized by
$$
\phi(\ell_1+\ell_2) \leq \phi(\ell_1)+\phi(\ell_2).
$$
The same argument as before shows that this induces a self-attraction of the chain. In the attractive case, we can, without loss of generality, make the following further assumption on $\phi$:
\begin{equation}\label{eq-sublinear}
\lim_{\ell\to\infty} \frac{\phi(\ell)}\ell = 0.
\end{equation}
Indeed, if $\phi$ was growing super-linearly, then the energy would always dominate the entropy (which obviously grows linearly with the chain length), and the polymer would always concentrate on two vertices. On the other hand, if there is a linear part in $\phi$, then it can be taken out, since the total length of the polymer is fixed, and thus the contribution of the linear part is independent of the polymer configuration. In the following, we shall always assume that~\eqref{eq-sublinear} is fulfilled when considering self-attractive interactions.
\end{itemize}

\begin{remark}
In this review, we only discuss attractive and repulsive potentials. Our results, however, are stable under perturbations. This is explained in details in~\cite{IV08}.

\end{remark}

Let now $F\in\Rd$ be the force applied at the free end-point. Our model is defined by the following probability measure on paths $\gamma=(\gamma(0),\ldots,\gamma(n))$,
$$
\Pnf{F}(\gamma) \defby \frac1{\Znf{F}}\, e^{-\Phi(\gamma)+\inprod{F}{\gamma(n)}}.
$$
Before finishing this section, let us mention how a few classical polymer models can be embedded into our framework.

\subsection{Some examples}
\label{ssec-examples}
\paragraph{The self-avoiding walk (SAW).} This is the standard model for self-repelling paths, with countless applications in polymer physics. It is defined by the probability measure giving equal weight to all length $n$ nearest-neighbor paths $\gamma$ on $\Zd$ satisfying the condition that no vertex is visited more than once. It is immediate to see that this corresponds to the particular choice
$$
\phi(\ell) =
\begin{cases}
\infty	&	\text{if $\ell\geq 2$}\\
0	&	\text{if $\ell\in\{0,1\}$}.
\end{cases}
$$

\paragraph{The Domb-Joyce model.} There are  numerous models of weakly self-avoiding walk, in which the hard-core self-avoidance condition characterizing the SAW is replaced by a soft-core penalty for multiple visits of a given vertex. The best-known such model is the Domb-Joyce model, in which the energy of a polymer chain $\gamma$ is given by
$$
\beta \sum_{0\leq i< j \leq n} \IF{\gamma(i)=\gamma(j)},
$$
for some $\beta>0$. It is easily verified that this corresponds to setting
$$
\phi(\ell) = \tfrac12\beta\ell(\ell-1).
$$
\paragraph{The discrete sausage.} This model has been studied in many papers under various names and interpretations. We chose this name in analogy with its continuous counterpart in which an underlying Brownian motion replaces the simple random walk. In this model, the energy of a polymer chain $\gamma$ is simply given by $\beta$ times the number of different vertices visited by $\gamma$. This corresponds to the particular choice
$$
\phi(\ell) =
\begin{cases}
\beta	&	\text{if $\ell\geq 1$}\\
0	&	\text{if $\ell=0$}.
\end{cases}
$$
The above interaction is obviously attractive.

\paragraph{The reinforced polymer.} This is a generalization of the previous model, formally analogous to a reinforced random walk. Let $(\beta_k)_{k\geq 1}$ be a non-increasing sequence of non-negative real numbers. The contribution to the energy of a given configuration $\gamma$ due to the $k^\text{th}$ visit at a given vertex is given by $\beta_k$. 
In other words, 
$$
\phi(\ell) = \sum_{k=1}^\ell \beta_k.
$$

\paragraph{The polymer in an annealed random potential.}
Another important example of polymer model with attractive self-interaction is given by a polymer in an annealed random potential. Let $(V_x)_{x\in\Zd}$, be a collection of i.i.d. non-negative random variables (the random potential). Given a realization $\theta$ of the environment, the quenched weight associated to a polymer configuration $\gamma$ is
\begin{equation}
\label{qweight}
\mathsf{w}_{\scriptscriptstyle\rm qu}^\theta (\gamma) \defby e^{-\sum_{i=0}^n V_{\gamma(i)}(\theta )}.
\end{equation}
It associates to each monomer the value of the potential at its location. The annealed weight of a polymer configuration corresponds to averaging the quenched weight w.r.t. the environment,
\begin{equation}
\label{aweight}
\mathsf{w}_{\scriptscriptstyle\rm an}(\gamma) \defby \mathbf{E} \mathsf{w}_{\scriptscriptstyle\rm qu}^\cdot (\gamma).
\end{equation}
Physically, this corresponds to a situation in which both the polymer and the environment have had time to reach equilibrium. It is easy to check that the probability measure associated to the annealed weights corresponds to choosing
$$
\phi(\ell) = -\log\bfE e^{-\ell V}.
$$

\subsection{The inverse correlation length}
In our analysis, a crucial role is played by the 2-point function and the associated inverse correlation length. The \define{2-point function} is defined, for $\lambda\in\bbR$ and $x\in\Zd$, by
$$
G_\lambda(x) \defby \sum_{\gamma:0\to x} e^{-\Phi(\gamma) - \lambda |\gamma|}\ \defby  
\sum_{\gamma:0\to x} W_{\lambda}  (\gamma ),
$$
where the sum runs over all nearest-neighbor paths $\gamma$ from $0$ to $x$ (of arbitrary length), and $|\gamma|$ denotes the length of $\gamma$. It is easy to see that $G_\lambda(x)$ is finite (for all $x\in\Zd$) as soon as $\lambda>\lambda_0$, where
$$
\lambda_0 \defby \lim_{n\to\infty} \frac1n \log \sum_{\substack{\gamma: |\gamma|=n\\ \gamma(0)=0}} e^{-\Phi(\gamma)}
$$
is essentially the free energy per monomer associated to a (free) polymer chain.
The above limit is well defined by sub-  (super-) additivity in the repulsive (attractive) case.
It can also be shown that $\lambda_0\in(0,\infty)$ always holds, with $\lambda_0=\log(2d)$ in the attractive case~\cite{Flury,IV08} 

Moreover, the 2-point function is infinite (for all $x\in\Zd$) when $\lambda<\lambda_0$, so that there is a transition at the critical value $\lambda_0$. Indeed, $\lambda <\lambda_0$ readily implies 
divergence of the bubble diagram 
$\sum_x G_\lambda (x)^2 = \infty$. In view of, \eg, (A.2) in \cite{IV08}, this already implies divergence
of the two-point function in the repulsive case. On the other hand in the attractive case
$G_\lambda (0) =\infty$ as soon as $\lambda <\lambda_0$. The divergence of the two-point 
function for every $x$ follows, since in the attractive case  $G_\lambda (x) \geq 
H_\lambda (x) G_\lambda (0)$, where $H_\lambda (x)>0$ is the contribution of all the paths
$\gamma :0\mapsto x$ which are stopped upon arrival to $x$.

When $\lambda>\lambda_0$, not only is the 2-point function finite, but it is actually exponentially decreasing as a function of $x$. This exponential decay is best encoded in the \define{inverse correlation length} $\xl:\Rd\to\bbR$, defined by
$$
\xl(x) \defby \lim_{k\to\infty} -\frac1k \log G_\lambda([kx]),\qquad (\lambda>\lambda_0)
$$
where $[x]\in\Zd$ denotes the component-wise integer part of $x\in\Rd$. It can be proved that the inverse correlation length is well-defined, and is an equivalent norm on $\Rd$. It measures the directional rate of decay of the 2-point function, in the sense that
$$
G_\lambda(x) = e^{-\xi(n_x)\|x\| (1+o(1))},
$$
where $n_x\defby x/\|x\|$ and the function $o(1)$ converges to zero as $\|x\|$ goes to infinity.

An important object associated to the correlation length $\xl$ is the \define{Wulff} shape $\Kl$, defined by
$$
\Kl \defby  \setof {F\in\Rd} {\inprod{F}{x}\leq \xl (x),\, \forall x\in\Rd}.
$$
The name Wulff shape is inherited from continuum mechanics, where $\Kl$ is the equilibrium crystal shape once $\xl$ is interpreted to be a surface tension (\ie, $\Kl$ is the convex set with support function $\xl$). Alternatively, one can describe $\Kl$ in terms of polar norms, as was done, \eg, in \cite{FluryLD}: Introducing the polar norm
$$
\xi^*_\lambda (F) \defby \max_{x\neq 0}\frac{\inprod{F}{x}}{\xl (x)} = \max_{\xl (x) =1} \inprod{F}{x},
$$
we see that $\Kl$ can be identified with the corresponding unit ball,
$$
\Kl =\, \setof{F}{\xi_\lambda^* (F) \leq 1}.
$$
In general, the family $(K_\lambda)_{\lambda\geq \lambda_0}$ is an increasing (w.r.t. inclusion) sequence of convex subsets of $\Rd$ (actually, convex bodies as soon as $\lambda>\lambda_0$).

\subsection{Behavior of the correlation length as $\lambda\downarrow\lambda_0$}
As will be explained in the next section, the behavior of $\xl$ as $\lambda\downarrow\lambda_0$ has an important impact on the behavior of the polymer under stretching. In this respect, the following dichotomy between attractive and repulsive models can be established:
$$
\xln(x) \defby \lim_{\lambda\downarrow\lambda_0} \xl(x)
\begin{cases}
\equiv 0	&	\text{in the repulsive case,}\\
> 0		&	\text{in the attractive case.}\\
\end{cases}
$$
We refer to \cite{Flury,IV08} for the attractive case. The repulsive part is 
worked out in the Appendix.
The behavior of $\xln$ has an immediate impact on the limiting shape $\Kln$: $\Kln$ has non-empty
interior in the attractive case, whereas $\Kln = \{0\}$ in the repulsive case.

\section{Macroscopic behavior of the polymer}
Let us say that the polymer is in the \define{collapsed phase} if and only if
$$
\lim_{n\to\infty} \Pnf{F}(\tfrac1n|\gamma_n| > \epsilon) = 0,\qquad \forall\epsilon>0.
$$
This means that, in the macroscopic scaling limit, the polymer has no extension. We call the complementary phase the \define{stretched phase}.

The following theorem describes the macroscopic behavior of the polymer depending on the intensity of the applied force.
\begin{theorem}
\label{thm_phtrans}
$$
\text{The polymer is }
\begin{cases}
\text{in the collapsed state if $F\in\intKln$,}\\
\text{in the stretched state if $F\not\in\Kln$.}\\
\end{cases}
$$
\end{theorem}
In the repulsive case, we have seen that $\Kln=\{0\}$, and we can thus conclude that any non-zero applied force results in the polymer being macroscopically stretched. In the attractive case, however, $\intKln\neq\varnothing$, and the polymer remains macroscopically collapsed for small enough applied forces. This corresponds to the intuition: In the attractive case, the self-interaction favors the collapsed phase, while the force favors the stretched phase, and the phase transition is a consequence of this competition.

\smallskip
Notice that Theorem~\ref{thm_phtrans} does not describe the behavior of the polymer when the force belongs to the boundary of $\Kln$. This critical case is delicate and is still under investigation in the case of attractive self-interaction; see also Section~\ref{sec_OrderTransition}. This question in the case of repulsive self-interaction reduces to proving that the critical polymer (\ie, without any applied force) is sub-ballistic. This is still an open problem, even for the SAW in dimensions $2,3$ and $4$! Of course, it should be possible (but less interesting) to establish diffusivity in high enough dimensions using a suitable version of the lace expansion.

\section{Description of the stretched phase}
We now turn to the description of the polymer in the stretched phase. We treat simultaneously the cases of attractive and repulsive self-interactions, as the results in that phase are identical. The only 
assumption is thus that $F\not\in\Kln$, or, equivalently, that 
\begin{equation}
\label{Conjugate}
F\in\partial\Kl\quad \text{ for some}\quad  \lambda >\lambda_0 .
\end{equation}
The above $\lambda$ should be viewed as a conjugate parameter.

\subsection{Distribution of the end-point}
\label{ssec-endpoint}
The first natural problem is to determine the distribution of the free end-point of the polymer.

The following result guarantees that the polymer is indeed in the stretched regime: There exists
  $\bar{v}_F\in\Rd\setminus\{0\}$, $\epsilon >0$ and $c>0$ such that
$$
\bbP_n^F\left(\tfrac1n\gamma(n) \notin B_\epsilon(\bar{v}_F)\right) \leq e^{-c n}.
$$
This can then be strengthened into a strong local limit result. More precisely, there exists a rate function $J_F$ that is strictly convex and real analytic on $B_\epsilon(\bar{v}_F)$ and possesses a non-degenerate quadratic minimum at $\bar{v}_F$, and a positive, real analytic function $G$ on $B_\epsilon(\bar{v}_F)$ such that
$$
\bbP_n^F(\tfrac1n\gamma(n)=x) = \frac{G(x)}{n^{d/2}}\, e^{-nJ_F(x)}\,(1+o(1)),\qquad\text{as $n\to\infty$,}
$$
uniformly in $x\in\tfrac1n\Zd\cap B_\epsilon(\bar{v}_F)$.

\subsection{Description of the microscopic structure}

\subsubsection{Decomposition into irreducible pieces}
\begin{figure}
\begin{center}
\includegraphics[width=\textwidth]{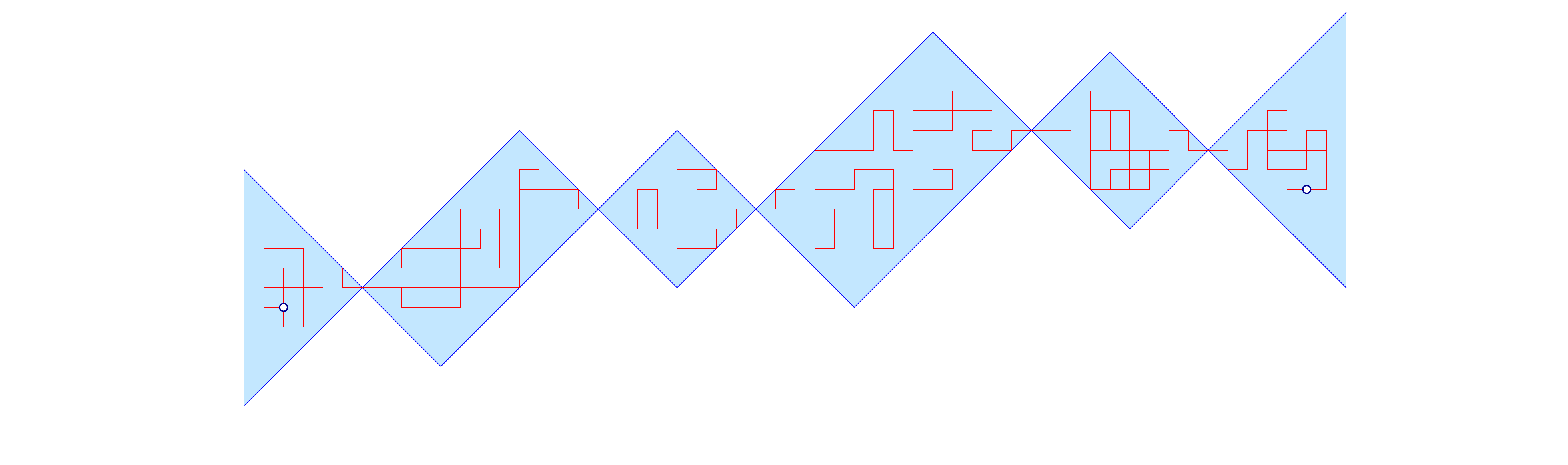}
\end{center}
\caption{The decomposition of a path $\gamma$ into irreducible pieces.}
\label{fig-irred}
\end{figure}

In the previous subsection, we have described the Gaussian fluctuations of the end-point of the polymer in the stretched phase. It is also natural, and physically relevant, as will be explained below, to describe the structure of the whole polymer chain at the microscopic scale. The crucial structural fact established in \cite{IV08} is that whenever the pulling force $F$ is strong in the sense of \eqref{Conjugate}, one is entitled to ignore all polymer configurations apart from those admitting a decomposition into irreducible pieces (see Fig.~\ref{fig-irred}),
\begin{equation}
\label{decomposition}
\gamma = \omega_{\rm\scriptscriptstyle L} \amalg \omega_1 \amalg\cdots\amalg\omega_m \amalg \omega_{\rm\scriptscriptstyle R},
\end{equation}
where $\amalg$ denotes concatenation. As Fig.~\ref{fig-irred} indicates, there are three types of irreducible pieces: the two extremal ones $\omega_{\rm\scriptscriptstyle L}, 
\omega_{\rm\scriptscriptstyle R}$, which could be viewed as boundary conditions,  and the bulk ones
 $\omega_1, \dots ,\omega_m$.
In either case, the size distribution of these pieces (\eg, diameter) has exponentially 
decaying tails \cite{IV08}. Consequently, the
boundary conditions have no impact on the large scale properties of such polymers and, for the 
sake of a more transparent notation, we shall ignore them
in the sequel, considering instead a simplified model with only bulk irreducible pieces present.
Let us describe the family $\Omega$ of the  latter.  Recall that 
we fixed the pulling force  $F$ and the conjugate 
parameter $\lambda >\lambda_0$, such that $F\in\partial\Kl$. For $\kappa\in(0,1)$ define a positive
cone $\cone \defby \setof{x\in\Zd}{|\inprod{x}{F}| > \kappa \xl(x) }$. It is convenient to choose  $\kappa$ sufficiently small to ensure that $\cone$ contains one of the neighbors of the origin.

Given a path $\gamma=(\omega (0),\ldots,\omega (l))$, we say that
 $\omega (k)$ is a 
\define{cone-point} of $\omega$ if $k=1, \dots, l-1$, and 
$$
(\omega (0),\ldots,\omega(k-1)) \subset \omega (k) - \cone \text{ and } (\omega (k+1),\ldots,\omega (l)) \subset \omega (k) + \cone.
$$
We say that a path $\omega$ is \define{irreducible} if it does not contain cone points and, in addition
is \define{cone-confined} in the following  sense, 
\begin{equation}
 \label{ConeConfine}
\omega \subseteq \lb \omega (0) + \cone\rb \cap \lb \omega (l) -\cone\rb .
\end{equation}
We denote by $\Omega$ the set of irreducible paths, identifying paths differing only by $\bbZ^d$ shifts. 

Note that the set $\Omega$ is adjusted to $F$ and, accordingly, to $\lambda$. However, the definitions
are set up in such a way that one may use the same set $\Omega$ for all 
pulling forces in a neighborhood
of $F$. This is important for application of analytic local limit theory, which is behind most
of the results we derive. There are two canonical variables associated with 
$\omega = (\omega (0) ,\dots ,\omega (l))\in\Omega$. The first one is the number of steps, $|\omega |=l$.
The second one is the displacement, $D (\omega ) = \omega (l ) - \omega (0 )$. Then, 
as was proved in \cite{IV08},
\[
 \bbP_\lambda^F (\omega ) = {\rm exp}
\lbr\phi (1) -\Phi (\omega ) -\lambda |\omega | +\inprod{F}{D (\omega )}\rbr
\]
is a probability measure on $\Omega$ with exponentially decaying tails: There 
exist $\nu_1 ,\nu_2 > 0$, such that uniformly in $l\geq 0$,
\begin{equation}
 \label{ExpTails}
\bbP_\lambda^F \lb | \omega | > l \rb
+\bbP_\lambda^F\lb  | D (\omega )|  > l \rb
 \leq \nu_1 {\rm e}^{-\nu_2 l} .
\end{equation}
 Consider the product measure 
$\otimes \bbP_\lambda^F$  on 
the space of countable strings\linebreak$\lb\omega_1 ,\omega_2 ,\dots \rb$ of elements 
$\omega_i\in\Omega$.   Recall our convention to 
 ignore boundary pieces in \eqref{decomposition}. Then one can express the canonical
partition function  $Z_n^F$ in terms of the grand-canonical measure $\otimes \bbP_\lambda^F$
as follows:
\begin{equation}
 \label{CPF1}
Z_n^F = {\rm e}^{\lambda n}\, \sum_N \otimes\Prob_\lambda^F \Bigl( \sum_{i=1}^N 
|\omega_i | = n\Bigr) .
\end{equation}
Similarly, for $x\in\cone$, one can express two point functions, 
\begin{equation}
 \label{CPF2}
G_\lambda (x) = {\rm e}^{-\inprod{F}{x}}\,
\sum_N
\otimes\Prob_\lambda^F \Bigl( \sum_{i=1}^N  D (\omega_i ) = x\Bigr) .
\end{equation}
In view of the Cramer-type condition \eqref{ExpTails}, relations \eqref{CPF1} and \eqref{CPF2} pave the way for a comprehensive local limit description of the microscopic geometry of polymer chains in the corresponding canonical ensembles. Applications for statistics of a general class of local observables are discussed in \cite{IV08}. For example one readily infers that, under $\bbP_n^F$, typical polymers $\gamma$ are composed of $\lb 1 + O (1/\sqrt{n})\rb n/\bbE_n^F |\omega |$ irreducible pieces whose lengths are at most $O (\log n )$.  In particular, the following invariance principle holds.

\subsubsection{An invariance principle}
As before, let $\lambda >\lambda_0$ and $F\in\partial\Kl$.  Then \cite{IV08} there exists an $\bbR^d$-neighborhood $\calU$ of $F$ such that the function $\mu = \mu (H)$, which is defined through the relation $H\in\Km$ (note that in this notation $\mu (F ) =\lambda$), is analytic on $\calU$ and, furthermore, the Hessian $\Sigma_F \df {\rm d}^2 \mu (F) $ is non-degenerate.

In terms of $\mu$ the average displacement per step is given by $\bar{v}_F = \nabla \mu (F )$~\cite{IV08}. Now, for a given $\gamma$ with $|\gamma | =n$, let $\gamma = \omega_1 \amalg\dots\amalg\omega_m$ be its irreducible decomposition. Again, recall that for simplicity we ignore boundary pieces in \eqref{decomposition}. With the irreducible decomposition at hand, we define the interpolated trajectory $\frg_n =\frg_n [\gamma ] : [0,1]\to\bbR^d$ as follows:
(1) Let $\frG_n :[0,n]\to \bbR^d$ be the linear interpolation through the space-time points
\[
 \bigl(0,0\bigr), \bigl( |\omega_1| ,D (\omega_1 )\bigr), \bigl(  |\omega_1| + |\omega_2|, 
 D (\omega_1 )+ D (\omega_2 )\bigr),\dots ,\bigl( n , \sum_1^m D(\omega_i )\bigr) .
\]

(2) For $t\in [0,1]$ define $\frg_n (t) = \lb \frG_n (tn )- tn\bar{v}_F\rb/\sqrt{n}$. 

Then $( \frg_n ,\bbP_n^F )$ weakly converges to the law of 
 $\sqrt{\Sigma_F} B_t$, $t\in [0,1]$, where
$B$ is the standard Brownian 
motion on $\bbR^d$.

\subsubsection{Pincus' blobs}
\begin{figure}
\begin{center}
\scalebox{.5}{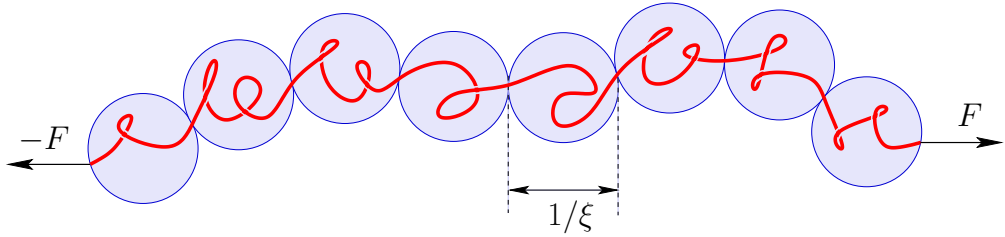}
\end{center}
\caption{Pincus' blobs picture (adapted from~\cite{deGennes}).  Under tension, the polymer decomposes into a succession of independent ``blobs'' of a size given by the correlation length $1/\xi$. Inside each blob, the corresponding piece of polymer has the same scaling properties as a critical (\ie, without applied forces) polymer.}
\label{fig-Pincus}
\end{figure}

In 1976, in order to derive various scaling properties of stretched polymers (which he modeled by SAW), Pincus introduced a heuristic description of the polymer now known as \define{Pincus' blobs picture}. The latter assumes that, in the stretched phase, the polymer's structure is that of a string of ``blobs'', which possess (in the terminology used in the present paper) the following characteristic properties:
\begin{itemize}
\item statistical independence;
\item sizes of the order of the correlation length $\lambda(F)$;
\item the same scaling properties as non-stretched polymers.
\end{itemize}
It seems likely that the decomposition into irreducible pieces described in the previous subsection should yield the rigorous counterpart to this blobs picture, but a precise identification would require substantially more work.

On the positive side, we already know that the irreducible pieces are statistically independent (modulo the fact that their total length is fixed to $n$).

Concerning the sizes of the irreducible pieces, our estimates do not allow us to say that they are of the order of the correlation length. One does have a weak version of such a claim: Coarse-graining the polymer at the correlation scale yields an object that does indeed admit a decomposition into irreducible pieces of the right scale, but our estimates are too poor to be converted into estimates for the underlying microscopic object. This seems however to be a purely technical issue, and, at least for a subclass of these polymer models, one should be able to improve on the latter estimates.

The most difficult aspect seems to be the last one. It is not actually really clear to us what the statement exactly means, but it seems to require a good understanding of the collapsed phase, which seems to be quite difficult in general, although this might be easier in the attractive case (using, \eg, the technology described in~\cite{Sznitman}), or in high enough dimensions (using, \eg, the lace expansion).

A related question of interest is to understand how $\bar{v}_F$ scales with the applied force $F$. Pincus' conjecture is that $\bar{v}_F \propto F^{\chi}$, where $\chi=(1/\nu)-1$ with $\nu$ the exponent characterizing the growth of the free polymer, $\bbE_n(\|D\|) \propto n^{\nu}$. Of course, the determination (and even existence!) of the exponent $\nu$ itself is an open problem in general (\eg, for the SAW below the critical dimension).

\smallskip
In any case, providing a rigorous version of this heuristic picture would be very interesting for two reasons: First, its validity is taken as a basic assumption in many works in polymer physics (see~\cite{deGennes,RubinsteinColby} and references therein), and, second, its validity, especially in the regime when $Fn^\nu$ is \emph{not} very large (in our work, we only consider the regime in which $n>n_0(F)$, for some fixed $F$, so our results shed no light on this problem yet), is not considered obvious even from the point of view of Theoretical Physics~\cite{Maggs_et_al}.

\section{Order of the collapsed/stretched phase transition}
\label{sec_OrderTransition}

As we have seen above, there is a non-trivial phase transition between a collapsed and a stretched phase in the case of polymers with self-attractive interactions. In this section, we present some preliminary results describing in more details the behavior of the polymer at this transition.

The problem investigated is the determination of the order of the phase transition. Let us fix a direction $h\in\partial\Kln$ and consider a force of the form $F=\alpha h$, $\alpha\in\bbR^+$. The criterion used is the behavior of the macroscopic position of the free end-point of the polymer which, as we have seen, is given by $0$ in the collapsed phase ($\alpha<1$), and by $\bar{v}_F\neq 0$ in the stretched phase ($\alpha>1$). We say that the transition is of \define{first order} if the position is discontinuous at the transition: $\lim_{\alpha\downarrow 1} \bar{v}_F\neq 0$, and is of \define{continuous} otherwise. (The order turns out not to depend on the chosen direction $h$.)

This question has already been investigated in the Physics' literature~\cite{MeGr} (see also the earlier work~\cite{GrPr}) for the particular case of the discrete sausage, $\Phi=\beta\,\#\setof{\gamma(k)}{0\leq k \leq n}$ (which is equivalent to a random walk among annealed killing traps). The conclusion drawn by the authors for this particular model are the following: The transition is of continuous in the one-dimensional case, but of first order in all higher dimensions. Their argument was analytic in $d=1$, but relied entirely on numerical evidences for $d\geq 2$.

In the case of the Wiener sausage, a lot of information has been extracted about the behavior of the path in the collapsed regime: In dimension $1$, \cite{EiLa,Po1995} (and also~\cite{Po1997} for soft obstacles) provide a detailed description of the path for all subcritical forces. In higher dimensions, the results seem to be restricted to the case $F=0$, see~\cite{Sz} for the two-dimensional Wiener sausage and~\cite{Po1999} for the higher-dimensional Wiener sausage; similar results for the two-dimensional discrete sausage have been obtained in~\cite{Bo}. The order of the transition itself does not seem to have been investigated rigorously in these works.

It turns out, however, that it is easy to verify that the transition is continuous for the one-dimensional discrete sausage.
 Namely,
\begin{theorem}
Assume that $d=1$ and $\phi(\ell)=\beta\IF{\ell\geq 1}$ ($1d$ discrete sausage).
Let $h\in\partial\Kln$. There exist $a_1,a_2,a_3>0$ such that, for any $\epsilon>0$,
$$
\Pnf{\alpha h}\left( \tfrac1n D(\gamma) > \epsilon \right) \leq e^{-a_1 \epsilon^2 n},
$$
for all $n\geq a_2\beta\epsilon^{-3}\vee a_2\epsilon^{-2}|\log\epsilon|$ and all $\alpha>1$ satisfying $|\alpha-1|\leq a_3\epsilon/h$. In particular, $\lim_{\alpha\downarrow 1} \bar{v}_{\alpha h}=0$ and the transition is continuous.
\end{theorem}
\begin{proof}
Let us write $\Znx^{\lambda,F} = \sum_{\substack{\gamma:0\to x\\|\gamma|=n}} e^{-\Phi(\gamma) - \lambda |\gamma| + \inprod{F}{D(\gamma)}}$. We use the following convention: We drop $F$ and/or $\lambda$ from the notation when they take value $0$, and similarly drop $x$ or $n$ from the notation when the corresponding constraint is removed. For example, we write $\Zxf{\lambda} = \sum_{\gamma:0\to x} e^{-\Phi(\gamma) - \lambda|\gamma|}$. Also, given a partition function $Z$ and a family of paths $\calA$, we write $Z[\calA]$ for the corresponding partition function restricted to paths in $\calA$.

We can assume that $h>0$ (\ie, $h=\xi_{\lambda_0}(1)$).
We first bound from above the numerator of the probability:
\begin{multline*}
\Znf{\alpha h}\left[ D(\gamma) > \epsilon n \right]
=
\sum_{x>\epsilon n} e^{\lambda_0 n}\, \Znx^{\lambda_0,\alpha h}
=
\sum_{x>\epsilon n} e^{\lambda_0 n+\alpha hx}\, \Znx^{\lambda_0}\\
=
\sum_{x>\epsilon n} e^{\lambda_0 n+\alpha hx}\, \frac{\Znx^{\lambda_0}}{\Zxf{\lambda_0}}\, \Zxf{\lambda_0}
\leq
e^{\lambda_0 n}\sum_{x>\epsilon n} e^{(\alpha-1) hx}\, \frac{\Znx^{\lambda_0}}{\Zxf{\lambda_0}},
\end{multline*}
since $\alpha h x = \xi_{\lambda_0}(x) + (\alpha-1) h x$ and $\Zxf{\lambda_0} \leq e^{-\xi_{\lambda_0}(x)}$ (by sub-additivity).
Moreover,
$$
\Znx^{\lambda_0}
=
\ESRW\left[ e^{-\Phi(X_{[0,n]})}, X_n=x \right]
\leq
e^{-\beta x}\, \PSRW\left( X_n=x \right),
$$
while, using $\tau_x=\min\setof{n\geq 1}{X_n=x}$,
\begin{align*}
\Zxf{\lambda_0}
=
\sum_{n\geq 0} \ESRW\left[ e^{-\Phi(X_{[0,n]})}, X_n=x \right]
&\geq
e^{-\beta x}\,\sum_{n\geq 0} \PSRW\left( \tau_{-1} > \tau_x, \tau_x=n \right)\\
&\geq
e^{-\beta x}\, \PSRW\left( \tau_{-1} > \tau_x \right)
\geq
\frac {1}{x+1} e^{-\beta x}.
\end{align*}
We thence have, for some $c_1>0$,
\begin{multline*}
\Znf{\alpha h}\left[ D(\gamma) > \epsilon n \right] \leq e^{\lambda_0 n} (n+1)\sum_{x>\epsilon n} 
e^{(\alpha-1) hx} \PSRW\left( X_n=x \right)\\
\leq e^{\lambda_0 n}\, (n+1)\sum_{x>\epsilon n} e^{(\alpha-1) hx-c_1 x^2/n} \leq e^{\lambda_0 n}\, e^{-c_2 \epsilon^2 n},
\end{multline*}
provided that $(\alpha-1)h\ll\epsilon$ and $n\gg \epsilon^{-2}|\log n|$.

Let us now turn to the denominator. We have, with $M=(n/\beta)^{1/3}$,
\begin{multline*}
\Znf{\alpha h} \geq \Znf{\alpha h}\left[ 0\leq \gamma(k) < M,\; \forall k\leq n \right]\\
\geq
e^{-\beta M} e^{\lambda_0 n} \PSRW (0 \leq X_k < M, \forall 1\leq k\leq n )
\geq
e^{\lambda_0 n}\, e^{-c_3\beta^{2/3}n^{1/3}},
\end{multline*}
since $\PSRW (0 \leq X_k < M, \forall 1\leq k\leq n ) \geq \exp(-cn/M^2)$, for some $c>0$.
Finally,
$$
\Pnf{\alpha h}\left( D(\gamma) > \epsilon n \right) \leq e^{-(c_2\epsilon^2 -c_3\beta^{2/3} n^{-2/3})n} \leq e^{-c_4
\epsilon^2 n},
$$
provided $n$ is so large that $n^{-2/3}\beta^{2/3}\ll\epsilon^2$.
\end{proof}
Although the previous result confirms the corresponding prediction in~\cite{MeGr}, it seems that among one-dimensional self-attractive polymer models, this behavior is pathological. Indeed, as the following heuristic argument indicates, the generic situation in dimension $1$ should be that the transition is continuous at high temperatures only, and becomes of first order at low temperatures.

This can be easily understood (at a heuristic level at least) by considering what happens for the variant of the discrete sausage in which $\phi_2(\ell) = \beta \IF{\ell\geq 1} + \beta\IF{\ell\geq 2}$.

The discrete sausage is equivalent to a random walk among annealed killing traps in one dimension. What made the above proof work is that, for the walk to reach a distant point $x$, it must be the case that there are no obstacles between $0$ and $x$. But once this happens, the random walk can reach $x$ in a diffusive way at virtually no additional cost.

In the variant, however, the situation is completely different, at least when $\beta\gg 1$. Indeed, it is equivalent to a random walk among annealed i.i.d. traps coming in two colors: each vertex is independently occupied by a black trap with probability $p=1-e^{-\beta}$, by a white trap with probability $p$, by both with probability $p^2$, or is empty. The random walks dies if it steps on a black trap, or if it visits at least twice a white trap. By the same reasoning as before, we see that there must be no black traps in the interval between $0$ and $x$. Now there are two possible strategies: either we remove all traps altogether in this interval, or we force the walk to move ballistically (since it cannot visit more than once each remaining (white) trap, and has to cross them to reach $x$). A simple computation shows that, at large $\beta$, the second alternative is much less costly, and thus the critical pinned polymer should be ballistic. The situation is opposite at small $\beta$, and in this case the polymer should travel diffusively in an interval essentially free of traps.

This is reminiscent of what happens for the one-dimensional Brownian bridge among pointlike obstacles, for which it is proved~\cite{Po1998} that it is favorable for the Brownian 
motion to go ballistically to its final point when the obstacles are soft, but not in the case of hard obstacles.

\medskip
Let us turn now to our (still preliminary) results in dimensions $2$ and higher. In the one-dimensional case, we essentially reduced the analysis of the order of the transition to determining whether typical paths $\gamma:0\to x$ are ballistic at $\lambda=\lambda_0$ (for large $\|x\|$). Such a relationship still holds in higher dimensions, and is the core of our approach to this problem. Namely, the following argument shows that ballistic behavior of typical polymers $\gamma:0\to x$ at $\lambda=\lambda_0$ implies that the transition is of first order.

Let $h\in\Kln$, $\alpha>1$ and $\rho>0$. Using the fact that $\xi_{\lambda_0}(x) \geq \inprod{h}{x}$, for all $x\in\Rd$, and proceeding as above, we can write
\begin{align*}
\Znf{\alpha h}\left[ \|D(\gamma)\| \leq \rho n \right]
&=
\sum_{x:\,\|x\|\leq\rho n} e^{\lambda_0 n+\alpha\inprod{h}{x}}\, \Znx^{\lambda_0}
\leq
\sum_{x:\,\|x\|\leq\rho n} e^{\lambda_0 n+\alpha\inprod{h}{x}}\, \Zxf{\lambda_0}\\
&\leq
\sum_{x:\,\|x\|\leq\rho n} e^{\lambda_0 n+\alpha\xi_{\lambda_0}(x)}\, \Zxf{\lambda_0}
\leq
e^{\lambda_0 n+(\alpha-1) \kappa \rho n + o(n)},
\end{align*}
where $\kappa=\sup_{\|x\|=1} \xi_{\lambda_0}(x)$.

On the other hand, let $x_h$ be the unit vector dual to $h$. Assume that one can show the following ballisticity statement: There exists $\bar\rho>0$ such that, with $x=[\bar\rho n x_h]$,
\begin{equation}
\label{eq_BallisticCritical}
\bbP_{x}^{\lambda_0} (|\gamma|=n) \geq e^{-o(n)}, \qquad \text{as $n\to\infty$.}
\end{equation}
Then, using this $x$, we would have, since $\Zxf{\lambda_0}\asymp e^{-\xln(x)}$ and $\inprod{h}{x}=\xln(x)$,
\begin{multline*}
\Znf{\alpha h}
\geq
e^{\lambda_0 n} e^{\alpha\inprod{h}{x}}\, \frac{\Znx^{\lambda_0}}{\Zxf{\lambda_0}}\, \Zxf{\lambda_0}
\geq
e^{\lambda_0 n} e^{(\alpha-1) \xln (x) - o(n)}\,\bbP_{x}^{\lambda_0} (|\gamma|=n)\\
\geq
e^{\lambda_0 n} e^{(\alpha-1) \xln (x_h) \bar\rho n -o(n) },
\end{multline*}
which would allow us to conclude that the transition is of first order, since this would imply that, for small enough $\rho$,
$$
\Pnf{\alpha h}\left( \| D(\gamma) \| \leq \rho n \right) \leq e^{-c(\alpha) n},
$$
for any $\alpha>1$ and $n$ large enough.

\medskip
The higher dimensional problem is thus reduced to proving~\eqref{eq_BallisticCritical}. One natural way to proceed is to try to extend the Ornstein-Zernike analysis developed in~\cite{IV08} for the case $\lambda>\lambda_0$ to the case $\lambda=\lambda_0$. This requires a slightly different coarse-graining, and a substantially refined argument and is still under progress. One part of the argument is already complete and explains a fundamental difference between dimensions $1$ and higher.  It follows from the coarse-graining argument alluded to above that there always exists a ballistic random tube of (in average) bounded cross-section connecting $0$ and $x$ inside which the polymer has to remain. In dimensions $2$ and higher, this forces the polymer to move ballistically, since staying inside this tube longer than needed has a high entropic cost. In dimension $1$, however, as there is no transverse direction in which to escape the tube, this entropic cost disappears and the polymer can behave diffusively (although it does not necessarily do so, as discussed above).

\section{Semi-directed polymer in a quenched random environment}
\label{sec_QuenchedDisorder}
As already mentioned, the case of a polymer in an annealed random potential falls into the general framework considered in the present paper. In this section, following a  work in progress, we present some partial new results concerning the case of a stretched polymer in a quenched weak disorder. The full account will be published elsewhere.

The strength of the disorder will be modulated by an additional parameter $\beta\geq 0$. Accordingly, we adjust the quenched and the annealed weights in \eqref{qweight} and \eqref{aweight} as
\[
 \wq (\gamma ) = {\rm e}^{-\beta\sum_{i=0}^n V_{\gamma (i )}(\theta  ) 
- \lambda |\gamma | }\quad\text{and}
\quad \wa (\gamma ) = \bfE \wq (\gamma ) .
\]
Given $N\in\bbN$ define
\begin{equation*}
 \calH_N^-\, 
=\, \setof{ x = (x_1 ,\dots ,x_{d})\in\bbZ^{d}}{ x_{1} < N}
\end{equation*}
and its outer vertex boundary $\calL_N =\partial\calH_N^-$. Consider the family $\calD_N$ of nearest neighbor paths from the origin $0$ to $\calL_N$, and define the corresponding quenched and annealed partition functions,
\[
Z_N^\theta  = Z_N^\theta  (\lambda ,\beta ) = \sum_{\gamma\in \calD_N} \wq (\gamma )
\quad\text{and}\quad 
\bfZ_N  =\bfE Z_N^\theta  (\lambda ,\beta ).
\]
We shall  assume that the distribution of $V_x$-s has bounded support 
and that it satisfies $0\in {\rm supp} (V )\subseteq [0,\infty )$. The former
condition can be relaxed (see, \eg, \cite{Flury}), whereas the latter 
is just a normalization condition.

It has been recently proved by Flury \cite{Flury} and then 
reproved by Zygouras \cite{Zygouras}  that in four  and higher dimensions for any $\lambda  >\lambda_0 =\log (2d )$ the annealed and the quenched free energies are equal once $\beta$ is small enough. Namely, for all $\beta$ sufficiently small there exists $\xi =\xi (\lambda  ,\beta ) >0$ such that
\begin{equation}
 \label{Flury}
-\lim_{N\to\infty}\frac{1}{N}\log Z_N^\theta  \, =\,\xi\, =\, 
-\lim_{N\to\infty}\frac{1}{N}\log \bfZ_N.
\end{equation}
This is an important result: In sharp contrast with models of directed polymers, the model of semi-directed polymers does not have an immediate underlying martingale structure and, subsequently one has to look for different (and arguably more intrinsic) ways to study it.

We shall take \eqref{Flury} as the starting point. Since, as explained in preceding sections, we fully control the annealed weights, there are some consequences at hand. First of all,  it immediately follows from Chebyshev inequality that, given $\nu >0$, we may eventually ignore {\em in the quenched ensemble} sub-families $\calA_N\subset\calD_N$ which satisfy ${\mathbf Z}_N (\calA_N )\leq  {\rm e}^{-\nu N} {\mathbf Z}_N$.  In particular, we can continue
to restrict attention 
to paths $\gamma\in\calD_N$ which admit the irreducible decomposition \eqref{decomposition} with only bulk irreducible paths present.  Next, the constant $\xi$ in \eqref{Flury} can be identified in terms of inverse correlation length as $\xi = \xi_\lambda (\sfe_1 )$, where $\sfe_1$ is the unit vector in the first coordinate direction.  More precisely, as in \eqref{CPF2}, we can express the annealed partition function as
\begin{equation}
\label{ZanN}
\bfZ_N  = \sum_{x\in\calL_N} G_\lambda (x)  = {\rm e}^{-N\xi }
\sum_M
\otimes\Prob_\lambda^{\xi\sfe_1}  \Bigl( \sum_{i=1}^M \inprod{D (\omega_i )}{\sfe_1} = N\Bigr) . 
\end{equation}
Here is our refinement of the result by Flury and Zygouras,
\begin{theorem}
\label{Thm:limit}
Let $d\geq 4$.  Then for every $\lambda >\lambda_0$ there exists $\beta_0 =\beta_0 (\lambda ,d)$, such that for every $\beta \in [0 ,\beta_0 )$ the limit; 
\begin{equation}
 \label{eq:limit}
\Xi^\theta =  \lim_{N\to\infty}
\Xi_N^\theta \df 
\lim_{N\to\infty} \frac{Z^\theta _N}{{\mathbf Z}_N}  \, \in\, (0,\infty ), 
\end{equation}
exists $\bfP$-a.s. and in $\bbL_2 (\Omega  )$.
\end{theorem}
Our second result gives a kind of justification to the prediction
 that semi-directed polymers should be diffusive at weak 
disorder: The random weights $\wq$ give rise to a (random) probability distribution $\mu_N^\theta $ on $\calD_N$. For a polymer $\gamma = (\gamma (0), \dots ,\gamma (\tau ))\in \calD_N$; $\tau = |\gamma |$, define $X (\gamma )$ as the $\bbZ^{d-1}$-valued transverse coordinate of its end-point; $\gamma (\tau ) = (N , X (\gamma ))$.
\begin{theorem}
\label{Thm:diffusive}
Let $d\geq 4$.   Then for every $\lambda >\lambda_0$ there exists $\hat \beta_0 =\hat \beta_0 (\lambda  ,d)$, such that for every $\beta \in [0, \hat\beta_0)$ the distribution of $X $ obeys the diffusive scaling with a non-random diffusivity constant $\sigma = \sigma (\beta ,\lambda ) >0$ in the following sense:
\begin{equation}
 \label{eq:diffusive}
\lim_{N\to\infty}\Xi_N^\theta \mu_N^\theta  \lb\frac{| X (\gamma )|^2 - \sigma^2 N}{N}\rb =0 
\quad \text{in $\bbL_2 (\Omega )$.}
\quad
\end{equation}
\end{theorem}
Above $\sigma (\beta ,\lambda )$   is precisely the (transversal) diffusivity constant of the corresponding annealed polymer model.  That is $\sigma$ is the Gaussian curvature of $\partial \Kl$ at $\xi(\sfe_1)$.

\smallskip 
We conclude this section with a brief comment regarding \eqref{eq:limit}.  One way to read \eqref{ZanN} is to claim that the quantity $\bft_\lambda (x)\df {\rm e}^{\xi x_1}G_\lambda (x)$ satisfies a $d$-dimensional renewal relation, 
\begin{equation}
\label{arenewal} 
\bft_\lambda (x) = \sum_y \bft_\lambda (y )\bfq_\lambda (x-y ) ,
\end{equation}
where 
\[
\bfq_\lambda ( z) = \bfE \frq_\lambda^\theta (z) \df \bfE\Bigl\{
{\rm e}^{\xi z_1} \sumtwo{\omega\in\Omega}{\calD (\omega )=z} \wq (\omega ) \Bigr\} .  
\]
The quenched version of \eqref{arenewal} is, 
\[
 \frt_\lambda^\theta  (x) = \sum_y \frt_\lambda^\theta  (y )\frq_\lambda^{\upsilon_y\theta} (x-y ), 
\]
where $\upsilon_y\theta$ is the corresponding shift of random environment.  Set $\frq^\theta_\lambda  = \sum_y \frq_\lambda^\theta (y)$.  Then the limit $\Xi^\theta$ in \eqref{eq:limit} is actually recovered through the following Ansatz, 
\begin{equation}
 \label{Ansatz}
\Xi^\theta = 1 +\sum_x   \frt^\theta_\lambda (x) \lb \frq^{\upsilon_x\theta}_\lambda -1\rb .
\end{equation}

\appendix
\section{Divergence of the correlation length for self-repulsive polymers}
\label{sec-newProof}
In this section, we provide a proof that the correlation length of self-repulsive polymers diverges as $F\downarrow 0$. This corrects the incomplete argument given in~\cite{IV08}\footnote{We are grateful to Jean Bérard for pointing out this flaw}.
By convexity, it is enough to consider $\xi_\lambda \df\xi_\lambda  (\sfe_1 )$.
\begin{proposition}
In the case of repulsive potentials,
$$
\xln = \lim_{\lambda\downarrow\lambda_0}\xi_\lambda =0 .
$$
\end{proposition}
\begin{proof}
The potential being repulsive, we know that $Z_n\geq e^{\lambda_0 n}$, for all $n\geq 1$, and thus can deduce that
\begin{equation}
\label{eq-divergenceatcriticality}
\sum_{x\in\Zd}G_{\lambda_0}(x) = \sum_{n\geq 1} Z_n e^{-\lambda_0 n} = \infty.
\end{equation}
We are going to show that a non-vanishing $\xln$ would contradict the latter statement.
Let us define $\alpha \df \lim_{\ell\to\infty}\frac{\phi (\ell )}{\ell}$.
\begin{lemma}
Assume that
\begin{equation}
\label{Assumption} 
\xi_{\lambda_0} = \lim_{\lambda\downarrow\lambda_0}\xi_\lambda >0 .
\end{equation}
Then, 
$
\alpha > \log 2d - \lambda_0 .
$
 In particular, there exist $c_1>0 $ and $ c_2<\infty$, such that
\begin{equation}
 \label{eq:phibound}
\lambda_0\ell  + \phi (\ell ) \geq \ell (c_1 +\log 2d) - c_2, 
\end{equation}
for all $\ell\in\bbN$.
\end{lemma}
\begin{proof}
Indeed, assume that $\alpha\leq  \log 2d - \lambda_0$. Then
$\phi (\ell )\leq \ell (\log 2d - \lambda_0)$. Consequently, for $\lambda = \lambda_0 +\epsilon$, 
\[
 (\lambda_0 +\epsilon )|\gamma | +\sum_x \phi (\ell_x (\gamma ))  \leq (\log 2d +\epsilon )
|\gamma | .
\]
As a result, $\xi_\lambda\leq \xi_\epsilon^{SRW}$, where the latter is the Lyapunov 
exponent of the simple random walk with killing rate $\epsilon$, known to tend to zero as 
$\epsilon$ tends to zero. 
\end{proof}
Let $G_K = [-K, \dots ,K]^d$ be the cube of radius $K$.
\begin{lemma}
Assume \eqref{Assumption}. Then for every $K$ there exists $\nu_1 = \nu_1 (K)$ such that
\begin{equation}
 \label{eq:Kd}
\sum_{\substack{\gamma(0)=0\\\gamma\subseteq G_K}}
W_{\lambda_0}  (\gamma )
\leq
\nu_1 (K).
\end{equation}
\end{lemma}
\begin{proof}
By \eqref{eq:phibound}, 
\[
W_{\lambda_0}  (\gamma ) \leq {\rm e}^{c_2 (2K+1 )^d - c_1|\gamma |}\lb\frac1{2d}
\rb^{|\gamma |} , 
\]
for every $\gamma$ such that $\gamma(0)=0$ and $\gamma\subseteq G_K$.
\end{proof}

Let $\calG_K$ be the set of paths $\gamma = (\gamma (0 ), \dots , \gamma (n))$ such that
$\gamma\setminus\gamma (n)\in G_K$ and $\gamma (n)\in\partial_{\rm out}G_K$.
\begin{lemma}
Assume \eqref{Assumption}. Then there exist $\nu_2 <\infty$ and $\nu_3 >0$ such that
\begin{equation}
\label{eq:GKSet}
\sum_{\gamma\in\calG_K} W_{\lambda_0}  (\gamma ) \leq \nu_2 {\rm e}^{-\nu_3 K} .
\end{equation}
\end{lemma}
\begin{proof}
The proof goes along the lines of the Hammersley-Welsh method as exposed in the 
book~\cite{SAWbook}. 
Let us say that $\gamma = (\gamma (0 ), \dots , \gamma (n))$ is a half-space path if
 $\gamma (0) = 0$ and 
\[
 \langle \gamma (\ell ) , \vec{e}_1\rangle_d < \langle \gamma ( n ) , \vec{e}_1\rangle_d , 
\]
for all $\ell < n = |\gamma |$.  The set of half-space paths with 
$\lb \gamma ( n ) , \vec{e}_1\rb_d = K$ will be denoted as $\calP_K$. Evidently, 
\eqref{eq:GKSet} will follow from
\begin{equation}
\label{eq:PK} 
\sum_{\gamma\in\calP_K} W_{\lambda_0}  (\gamma ) \leq \frac{\nu_2}{2d} 
{\rm e}^{-\nu_3 K} .
\end{equation}
In order to prove \eqref{eq:PK}, let us define yet another set of paths: We shall
say that $\gamma$ is a cylindrical path; $\gamma\in\calC_K$,  if 
$\gamma\in\calP_K$ and, in addition,
\[
 \langle \gamma ( \ell) , \vec{e}_1\rangle_d \geq 0 ,  
\]
for all $\ell = 0, \dots\ |\gamma |$.  The function
\[
 \bbC_\lambda (K) \df\sum_{\gamma\in\calC_K} W_\lambda (\gamma ) 
\]
is finite for all $\lambda >\lambda _0$. By our assumption on the potential $\Phi$, 
\[
  \bbC_\lambda (K + L)\geq {\rm e}^{- \phi (1) }
\bbC_\lambda (K)\bbC_\lambda (L) .
\]
Therefore, $K\mapsto {\rm e}^{- \phi (1) }\bbC_\lambda (K)$ is super-multiplicative.
By convexity of the inverse correlation length,
\[
 \lim_{K\to\infty}\frac1{K}\log \bbC_\lambda (K) \leq - \xi_\lambda .
\]
It follows that, 
\begin{equation}
 \label{eq:Cnot}
\bbC_\lambda (K ) \leq  {\rm e}^{\phi (1) - K\xi_\lambda (\vec{e}_1 )} ,
\end{equation}
uniformly in $K$ and in $\lambda >\lambda_0$. By 
monotone convergence, \eqref{eq:Cnot} holds at
$\lambda_0$ as well. 

Next, each path $\gamma\in\calP_K$ can be canonically 
(that is, by a series of reflections) mapped into a concatenation
of cylindrical paths, 
\[
 \gamma \mapsto \gamma_1\amalg\gamma_2\amalg\dots\amalg\gamma_r ,
\]
with $\gamma_i\in\calC_{M_i}$
such that,  $M_1\geq K$, $M_1 > M_2 >\dots >M_r$ and, in addition, 
\[
 W_{\lambda_0} (\gamma ) \leq {\rm e}^{r \phi (1)  }
\prod W_{\lambda_0} (\gamma_i  ).
\]
It follows that
\[
\sum_{\gamma\in\calP_K} W_{\lambda_0}  (\gamma ) \leq {\rm e}^{ \phi (1)}
\max_{M\geq K}\bbC_{\lambda_0} (M) \times  \prod_{M=1}^\infty
\lb 1+ {\rm e}^{\phi (1)}\bbC_{\lambda_0 } (M)\rb .
\]
The target inequality \eqref{eq:PK} follows now from \eqref{eq:Cnot}.
\end{proof}

In view of \eqref{eq:Kd} and \eqref{eq:GKSet} the coarse-graining procedure of~\cite{IV08} implies that once Assumption \eqref{Assumption} holds, $\sum_xG_{\lambda_0} (x)$ is (even exponentially) convergent, which is in contradiction with~\eqref{eq-divergenceatcriticality}.
\end{proof}

\bibliographystyle{acmtrans-ims}
\bibliography{IV-EBPRev}

\end{document}

%% file: fig-Pincus-2.pdf_t
\begin{picture}(0,0)%
\includegraphics{fig-Pincus-2.pdf}%
\end{picture}%
\setlength{\unitlength}{3947sp}%
\begingroup\makeatletter\ifx\SetFigFont\undefined%
\gdef\SetFigFont#1#2#3#4#5{%
  \reset@font\fontsize{#1}{#2pt}%
  \fontfamily{#3}\fontseries{#4}\fontshape{#5}%
  \selectfont}%
\fi\endgroup%
\begin{picture}(12558,2947)(-910,-5147)
\put(5908,-5002){\makebox(0,0)[lb]{\smash{{\SetFigFont{29}{34.8}{\rmdefault}{\mddefault}{\updefault}{\color[rgb]{0,0,0}$1/\xi$}%
}}}}
\put(11051,-3757){\makebox(0,0)[lb]{\smash{{\SetFigFont{29}{34.8}{\rmdefault}{\mddefault}{\updefault}{\color[rgb]{0,0,0}$F$}%
}}}}
\put(-737,-4036){\makebox(0,0)[lb]{\smash{{\SetFigFont{29}{34.8}{\rmdefault}{\mddefault}{\updefault}{\color[rgb]{0,0,0}$-F$}%
}}}}
\end{picture}%